\documentclass[12pt]{amsart}

\usepackage{amsmath,amssymb,amsthm,esint}
\usepackage{geometry}

\usepackage[pdfpagemode=UseNone,pdfstartview=FitH]{hyperref}

\newcommand{\abs}[1]{\left|#1\right|}
\newcommand{\bdry}[1]{\partial #1}
\newcommand{\closure}[1]{\overline{#1}}
\newcommand{\dint}{\ds{\int}}
\newcommand{\dist}[2]{\text{dist}\, (#1,#2)}
\newcommand{\ds}[1]{\displaystyle #1}
\newcommand{\eps}{\varepsilon}
\newcommand{\goodchi}{\protect\raisebox{2pt}{$\chi$}}
\newcommand{\hquad}{\hspace{0.08in}}
\newcommand{\interior}[1]{#1^\circ}
\newcommand{\ip}[2]{\left<#1,#2\right>}
\newcommand{\loc}{\text{loc}}
\newcommand{\norm}[2][]{\left\|#2\right\|_{#1}}
\renewcommand{\O}{\text{O}}
\renewcommand{\o}{\text{o}}
\newcommand{\PS}[1]{$(\text{PS})_{#1}$}
\newcommand{\QED}{\mbox{\qedhere}}
\newcommand{\restr}[2]{\left.#1\right|_{#2}}
\newcommand{\seq}[1]{\left(#1\right)}
\newcommand{\set}[1]{\left\{#1\right\}}
\newcommand{\strictsubset}{\subset \subset}
\newcommand{\vol}[1]{\left|#1\right|}

\newcommand{\B}{{\mathcal B}}
\newcommand{\D}{{\mathcal D}}
\newcommand{\M}{{\mathcal M}}

\newcommand{\R}{\mathbb R}

\DeclareMathOperator{\divg}{div}
\DeclareMathOperator{\supp}{supp}

\newenvironment{enumroman}{\begin{enumerate}

}{\end{enumerate}}

\newtheorem{lemma}{Lemma}[section]
\newtheorem{proposition}[lemma]{Proposition}
\newtheorem{theorem}[lemma]{Theorem}
\newtheorem{corollary}[lemma]{Corollary}

\theoremstyle{definition}
\newtheorem{definition}[lemma]{Definition}

\numberwithin{equation}{section}

\title{Existence and regularity of higher critical points 
in elliptic free boundary problems}
\date{}
\author{David Jerison and Kanishka Perera}
\thanks{The first author was supported 
by NSF grant DMS 1069225 and the Stefan Bergman Trust.}
\thanks{This work was completed while the second-named author 
was visiting the Department of Mathematics at the Massachusetts Institute 
of Technology, and he is grateful for the kind hospitality of the department.}
\address{David Jerison, Department of Mathematics, Massachusetts Institute of Technology, 77  Massachusetts Avenue, Cambridge, MA 02139.}
\email{jerison@math.mit.edu}
\address{Kanishka Perera, Department of Mathematical Sciences, 
Florida Institute of Technology, 
Melbourne, FL 32901.}
\email{kperera@fit.edu}
\subjclass[2010]{35R35, 35J20}
\keywords{superlinear elliptic free boundary problems, higher critical points, 
existence, nondegeneracy, regularity, variational methods, Nehari manifold}

\begin{document}


\begin{abstract}
Existence and regularity of minimizers in elliptic free boundary
problems have been extensively studied in the literature. We initiate
the corresponding study of higher critical points by considering a
superlinear free boundary problem related to plasma confinement. The
associated energy functional is nondifferentiable, and therefore
standard variational methods cannot be used directly to prove the
existence of critical points. Here we obtain a nontrivial generalized
solution $u$ of mountain pass type as the limit of mountain pass
points of a suitable sequence of $C^1$-functionals approximating the
energy. We show that $u$ minimizes the energy on the associated
Nehari manifold and use this fact to prove that it is
nondegenerate. We use the nondegeneracy of $u$ to show that it
satisfies the free boundary condition in the viscosity sense.
Moreover, near any free boundary point that has a measure-theoretic
normal, the free boundary is a smooth surface, and hence
the free boundary condition holds in the classical sense.
\end{abstract}

\maketitle

\section{Introduction}\label{sec: intro}

The existence and regularity of minimizers in elliptic free boundary
problems have been extensively studied for over four decades (see,
e.g.,
\cite{MR618549,MR732100,MR1044809,MR2145284,MR861482,MR990856,MR1029856,MR973745,MR1906591,MR2082392,MR2572253,MR1009785,MR0440187,MR1664689,MR2281453,MR1620644,MR1759450}
and the references therein). Let $\Omega$ be a smooth bounded domain
in $\R^N,\, N \ge 2$. A typical two-phase free boundary problem seeks
a minimizer of the variational integral
\[
\int_\Omega \left[\frac{1}{2}\, |\nabla u|^2 + \goodchi_{\set{u > 0}}(x)\right] dx
\]
among all functions $u \in H^1(\Omega) \cap C^0(\Omega)$ with prescribed values on some portion of the boundary $\bdry{\Omega}$, where $\goodchi_{\set{u > 0}}$ is the characteristic function of the set $\set{u > 0}$. A local minimizer $u$ satisfies
\[
\Delta u = 0
\]
except on the free boundary $\bdry{\set{u > 0}} \cap \Omega$, and
\[
|\nabla u^+|^2 - |\nabla u^-|^2 = 2
\]
on smooth portions of the free boundary, where $\nabla u^\pm$ are the
limits of $\nabla u$ from $\set{u > 0}$ and $\interior{\set{u \le
    0}}$, respectively. The existence and regularity of local
minimizers for this problem have been studied, for example, in Alt and
Caffarelli \cite{MR618549}, Alt, Caffarelli and Friedman
\cite{MR732100}, Weiss \cite{MR1620644,MR1759450}, Caffarelli, Jerison and Kenig
\cite{MR2082392}.

In the present paper we initiate the corresponding study of higher
critical points by considering a superlinear free boundary problem
related to plasma confinement (see, e.g.,
\cite{MR587175,MR1644436,MR1360544,MR1932180,MR0412637,MR0602544}). We
consider the functional
\[
J(u) = \int_\Omega \left[\frac{1}{2}\, |\nabla u|^2 + \goodchi_{\set{u > 1}}(x) - \frac{1}{p}\, (u - 1)_+^p\right] dx, \quad u \in H^1_0(\Omega),
\]
for $2 < p < \infty$ if $N = 2$ and $2 < p < 2N/(N - 2)$ if $N \ge 3$.
Here $H^1_0(\Omega)$ is the usual Sobolev space with the norm given
by 
\[
\norm{u}^2 = \int_\Omega |\nabla u|^2\, dx.
\]
The functional $J$ is
nondifferentiable and therefore standard variational methods cannot be
used directly to obtain a higher critical point. We will obtain our solution
as the limit of mountain pass points of a suitable sequence of
$C^1$-functionals approximating $J$. The crucial ingredient in the
passage to the limit is the uniform Lipschitz
continuity result, proved in Caffarelli, Jerison, and Kenig
\cite{MR1906591} (see Proposition \ref{Proposition 3}).

The mountain pass solution $u$ that we construct
is Lipschitz continuous in $\Omega$ and satisfies
the following Euler-Lagrange equations in a generalized sense 
(see Definition \ref{def: generalized}).
\begin{equation} \label{1}
\left\{\begin{aligned}
- \Delta u & = (u - 1)_+^{p-1} && \text{in } \Omega \setminus \bdry{\set{u > 1}}\\[10pt]
|\nabla u^+|^2 - |\nabla u^-|^2 & = 2 && \text{on } \bdry{\set{u > 1}}\\[10pt]
u & = 0 && \text{on } \bdry{\Omega},
\end{aligned}\right.
\end{equation}
where $w_\pm = \max \set{\pm w,0}$ are the positive and negative parts of $w$, 
respectively, $\nabla u^\pm$ are the limits of $\nabla u$ from 
$\set{u > 1}$ and $\interior{\set{u \le 1}}$, respectively.

It will follow from integration by parts that $u$ belongs to the Nehari-type 
manifold
\[
\M = \set{u \in H^1_0(\Omega) : \int_{\set{u > 1}} |\nabla u|^2\, dx = \int_{\set{u > 1}} (u - 1)^p\, dx > 0}.
\]
We will also show that $u$ minimizes the functional $J$ when restricted
to the Nehari manifold.  This minimizing property will permit us to deduce that the solution $u$ is non-degenerate in the following
sense.
\begin{definition}\label{def: degenerate}
A generalized solution $u$ of problem \eqref{1} is nondegenerate if
there exist constants $r_0,\, c > 0$ such that if $x_0 \in \set{u >
  1}$ and $r := \dist{x_0}{\set{u \le 1}} \le r_0$, then $u(x_0) \ge 1
+ c\, r$.
\end{definition}
This nondegeneracy makes it possible to apply
the regularity results in 
Lederman and Wolanski
\cite{MR2281453} showing
that $u$ satisfies the free boundary
condition in the viscosity sense and that near points of
the free boundary with a measure-theoretic normal, the free
boundary is a smooth surface and hence the free boundary condition holds in the
classical sense.

To formulate our results more precisely, recall 
the definition of the mountain pass solution. 
\begin{definition}[Hofer \cite{MR812787}]
We say that $u \in H^1_0(\Omega)$ is a mountain pass point of $J$ if the set $\set{v \in U : J(v) < J(u)}$ is neither empty nor path connected for every neighborhood $U$ of $u$.
\end{definition}
Let
\[
\Gamma = \set{\gamma \in C([0,1],H^1_0(\Omega)) : \gamma(0) = 0,\, J(\gamma(1)) < 0}
\]
be the class of continuous paths from $0$ to the 
set $\set{u \in H^1_0(\Omega) : J(u) < 0}$, and denote
\[
c^* = c^*(\Omega)  := \inf_{\gamma \in \Gamma}\, \max_{u \in \gamma([0,1])}\, J(u).
\]

Our main result is the following theorem.
\begin{theorem} \label{Theorem 1}  Let $\Omega$ be a smooth bounded
domain in $\R^N$. Let $2 < p < \infty$ if $N = 2$ and $2 < p < 2N/(N - 2)$ 
if $N \ge 3$. Then

a) $c^* = c^*(\Omega) >0$. 

b) There is a mountain pass point $u$ for the functional $J$ satisfying
$J(u) = c^*$, which is Lipschitz continuous and satisfies \eqref{1} in the 
generalized sense of Definition \ref{def: generalized}.

c) $u$ minimizes $\restr{J}{\M}$ and is nondegenerate.
\end{theorem}

\begin{corollary}[Lederman and Wolanski
\cite{MR2281453}] \label{Corollary 1}
Let $u$ be a nondegenerate generalized solution to 
problem \eqref{1} as in Theorem \ref{Theorem 1}. Then

a) $u$ satisfies the free boundary condition in the viscosity sense and

b) In a neighborhood of every free boundary point where the measure-theoretic 
normal exists, the free boundary is a $C^{1,\, \alpha}$-surface, and 
hence $u$ satisfies the free boundary condition in the classical sense.
\end{corollary}

We point out that the existence of a mountain pass solution is by no
means routine due to the severe lack of smoothness of $J$.  Indeed,
$J$ is not even continuous, much less of class $C^1$.  Note that for
the functional in which the discontinuous term $\chi_{\{u>1\}}$ is removed,
there is no difficulty in applying the mountain pass theorem (see
Flucher and Wei \cite{MR1644436} and Shibata \cite{MR1932180}).  We
believe that Theorem \ref{Theorem 1} is the first result in the
literature that establishes the existence of a higher critical point
and verifies its nondegeneracy in a free boundary problem of the type
\eqref{1}.

The paper is organized as follows. 
In Section \ref{sec: mountain pass} we construct a putative mountain pass 
solution $u$ as a limit of solutions to a regularized problem.
We prove that $u$ is a generalized solution in Section \ref{sec: generalized}.  
In Section \ref{sec: Nehari nondegeneracy} we show that $u$ belongs
to the Nehari-type manifold and that minimizers on the Nehari
manifold have the nondegeneracy property. 
In Section \ref{sec: final} we prove our main theorem by
showing that our solution $u$ does minimize over the Nehari
manifold.  We deduce Corollary \ref{Corollary 1}
and mention further questions about partial regularity of the free boundary.

\section{Limits of mountain pass solutions} \label{sec: mountain pass}


We approximate $J$ by $C^1$-functionals as follows. Let $\beta : \R
\to [0,2]$ be a smooth function such that $\beta(t) = 0$ for $t \le
0$, $\beta(t) > 0$ for $0 < t < 1$, $\beta(t) = 0$ for $t \ge 1$, and
$\int_0^1 \beta(s)\, ds = 1$. Then set
\[
\B(t) = \int_0^t \beta(s)\, ds,
\]
and note that $\B : \R \to [0,1]$ is a smooth nondecreasing function
such that $\B(t) = 0$ for $t \le 0$, $\B(t) > 0$ for $0 < t < 1$, and
$\B(t) = 1$ for $t \ge 1$. For $\eps > 0$, let
\[
J_\eps(u) = \int_\Omega \left[\frac{1}{2}\, |\nabla u|^2 + \B\left(\frac{u - 1}{\eps}\right) - \frac{1}{p}\, (u - 1)_+^p\right] dx, \quad u \in H^1_0(\Omega)
\]
and note that $J_\eps$ is of class $C^1$.

If $u$ is a critical point of $J_\eps$, then $u$ is a weak solution of
\begin{equation} \label{2}
\left\{\begin{aligned}
\Delta u & = \frac{1}{\eps}\, \beta\left(\frac{u - 1}{\eps}\right) - (u - 1)_+^{p-1} && \text{in } \Omega\\[10pt]
u & = 0 && \text{on } \bdry{\Omega},
\end{aligned}\right.
\end{equation}
and hence also a classical solution by elliptic regularity theory. By
the maximum principle, either $u > 0$ everywhere or $u$ vanishes
identically. If $u \le 1$ everywhere, then $u$ is harmonic in $\Omega$
and hence vanishes identically again. So, if $u$ is a nontrivial
critical point, then $u > 0$ in $\Omega$ and $u > 1$ in a nonempty
open subset of $\Omega$.

\begin{lemma}
$J_\eps$ satisfies the Palais-Smale compactness condition {\em \PS{}}, i.e., every sequence $\seq{u_j} \subset H^1_0(\Omega)$ such that $J_\eps(u_j)$ is bounded and $J_\eps'(u_j) \to 0$ has a convergent subsequence.
\end{lemma}

\begin{proof}
It suffices to show that $\seq{u_j}$ is bounded by a standard argument. We have
\begin{equation} \label{3.1}
J_\eps(u_j) = \int_\Omega \left[\frac{1}{2}\, |\nabla u_j|^2 + \B\left(\frac{u_j - 1}{\eps}\right) - \frac{1}{p}\, (u_j - 1)_+^p\right] dx = \O(1)
\end{equation}
and
\begin{multline} \label{3.2}
J_\eps'(u_j)\, v = \int_\Omega \left[\nabla u_j \cdot \nabla v + \frac{1}{\eps}\, \beta\left(\frac{u_j - 1}{\eps}\right) v - (u_j - 1)_+^{p-1}\, v\right] dx = \o(1) \norm{v},\\[7.5pt]
v \in H^1_0(\Omega).
\end{multline}
Since $\B$ and $\beta$ are bounded, writing
\[
u_j = (u_j - 1)_+ + [1 - (u_j - 1)_-] =: u_j^+ + u_j^-
\]
in \eqref{3.1} gives
\[
\int_\Omega \left[|\nabla u_j^+|^2 + |\nabla u_j^-|^2 - \frac{2}{p}\, (u_j^+)^p\right] dx = \O(1),
\]
and taking $v = u_j^+$ in \eqref{3.2} and using the Sobolev imbedding theorem gives
\[
\int_\Omega \Big[|\nabla u_j^+|^2 - (u_j^+)^p\Big]\, dx = \O(1)\, \|u_j^+\|.
\]
Combining the last two equations now gives
\[
\left(1 - \frac{2}{p}\right) \|u_j^+\|^2 + \|u_j^-\|^2 = \O(1) \left(\|u_j^+\| + 1\right),
\]
which implies that $\|u_j^\pm\|$, and hence $\norm{u_j}$, is bounded.
\end{proof}

Since $p < 2N/(N - 2)$, the Sobolev imbedding theorem implies
\[
J_\eps(u) \ge \int_\Omega \left[\frac{1}{2}\, |\nabla u|^2 - \frac{1}{p}\, |u|^p\right] dx \ge \frac{1}{2} \norm{u}^2 - C \norm{u}^p \quad \forall u \in H^1_0(\Omega)
\]
for some constant $C$ depending on $\Omega$. 
Since $p > 2$, then there exists a constant $\rho > 0$ such that
\[
\norm{u} \le \rho \implies J_\eps(u) \ge \frac{1}{3} \norm{u}^2.
\]
Moreover
\[
J_\eps(u) \le \int_\Omega \left[\frac{1}{2}\, |\nabla u|^2 + 1 - \frac{1}{p}\, (u - 1)_+^p\right] dx
\]
and hence, again because $p>2$, there exists 
a function $u_0 \in H^1_0(\Omega)$ such that $J_\eps(u_0) < 0 = J_\eps(0)$. 
Therefore the class of paths
\[
\Gamma_\eps = \set{\gamma \in C([0,1],H^1_0(\Omega)) : \gamma(0) = 0,\, J_\eps(\gamma(1)) < 0}
\]
is nonempty and
\begin{equation} \label{22}
c_\eps := \inf_{\gamma \in \Gamma_\eps}\, \max_{u \in \gamma([0,1])}\, J_\eps(u) \ge \frac{\rho^2}{3}.
\end{equation}

\begin{lemma} \label{Lemma 7}
$J_\eps$ has a (nontrivial) critical point $u^\eps$ at the level $c_\eps$.
\end{lemma}

\begin{proof}
If not, then there exists a constant $0 < \delta \le c_\eps/2$ and a
continuous map $\eta : \set{J_\eps \le c_\eps + \delta} \to
\set{J_\eps \le c_\eps - \delta}$ such that $\eta$ is the identity on
$\set{J_\eps \le 0}$ by the first deformation lemma (see, e.g., Perera
and Schechter \cite[Lemma 1.3.3]{MR3012848}). By the definition of
$c_\eps$, there exists a path $\gamma \in \Gamma_\eps$ such that
$\max_{\gamma([0,1])}\, J_\eps \le c_\eps + \delta$. Then
$\widetilde{\gamma} := \eta \circ \gamma \in \Gamma_\eps$ and
$\max_{\widetilde{\gamma}([0,1])}\, J_\eps \le c_\eps - \delta$, a
contradiction.
\end{proof}

\begin{lemma} \label{Lemma 1}
We have
\[
c_\eps \le c^*
\]
In particular, by \eqref{22}, $c^*>0$, and Theorem \ref{Theorem 1}
(a) holds.
\end{lemma}

\begin{proof}
Since $\B((t - 1)/\eps) \le \goodchi_{\set{t > 1}}$ for all $t$, $J_\eps(u) \le J(u)$ for all $u \in H^1_0(\Omega)$. So $\Gamma \subset \Gamma_\eps$ and
\[
c_\eps \le \max_{u \in \gamma([0,1])}\, J_\eps(u) \le \max_{u \in \gamma([0,1])}\, J(u) \quad \forall \gamma \in \Gamma. \QED
\]
\end{proof}

For $0 < \eps \le 1$, $u^\eps$ have the following uniform regularity properties.

\begin{lemma} \label{Lemma 3}
There exists a constant $C > 0$ such that, for $0 < \eps \le 1$,
\[
\norm{u^\eps} \le C.
\]
\end{lemma}

\begin{proof}
By Lemma \ref{Lemma 1},
\[
\int_\Omega \left[\frac{1}{2}\, |\nabla u^\eps|^2 + \B\left(\frac{u^\eps - 1}{\eps}\right) - \frac{1}{p}\, (u^\eps - 1)_+^p\right] dx \le c
\]
and hence
\begin{equation} \label{32}
\frac{1}{2} \int_\Omega |\nabla u^\eps|^2\, dx \le c + \frac{1}{p} \int_{\set{v_\eps > 0}} v_\eps^p\, dx,
\end{equation}
where $v_\eps = u^\eps - 1$. Testing \eqref{2} with $(u^\eps - 1 - \eps)_+$ gives
\begin{equation} \label{33}
\int_{\set{u^\eps > 1 + \eps}} |\nabla u^\eps|^2\, dx = \int_{\set{v_\eps > \eps}} v_\eps^{p-1}\, (v_\eps - \eps)\, dx.
\end{equation}
Fix $\lambda > 2/(p - 2)$. Multiplying \eqref{33} by $(\lambda + 1)/p \lambda$ and subtracting from \eqref{32} gives
\begin{multline} \label{34}
\frac{1}{2} \int_{\set{u^\eps \le 1 + \eps}} |\nabla u^\eps|^2\, dx + \frac{(p - 2) \lambda - 2}{2p \lambda} \int_{\set{u^\eps > 1 + \eps}} |\nabla u^\eps|^2\, dx\\[10pt]
\le c + \frac{1}{p} \int_{\set{0 < v_\eps \le \eps}} v_\eps^p\, dx + \frac{1}{p \lambda} \int_{\set{v_\eps > \eps}} v_\eps^{p-1}\, \big[(\lambda + 1)\, \eps - v_\eps\big]\, dx.
\end{multline}
The last integral is less than or equal to $\int_{\set{\eps < v_\eps < (\lambda + 1)\, \eps}} v_\eps^{p-1}\, \big[(\lambda + 1)\, \eps - v_\eps\big]\, dx$ and hence \eqref{34} gives
\[
\min \set{\frac{1}{2},\frac{(p - 2) \lambda - 2}{2p \lambda}} \int_\Omega |\nabla u^\eps|^2\, dx \le c + \frac{\eps^p \vol{\Omega}}{p} \left[1 + (\lambda + 1)^{p-1}\right].
\]
Since $\eps \le 1$, the conclusion follows.
\end{proof}

\begin{lemma} \label{Lemma 2}
There exists a constant $C > 0$ such that, for $0 < \eps \le 1$,
\begin{equation} \label{29}
\max_{x \in \Omega}\, u^\eps(x) \le C.
\end{equation}
\end{lemma}

\begin{proof}
Since $p < 2N/(N - 2)$, we have $N(p - 2)/2 < 2N/(N - 2)$. Fix $N(p - 2)/2 < q < 2N/(N - 2)$. Since
\[
- \Delta u^\eps = (u^\eps - 1)_+^{p-1} - \frac{1}{\eps}\, \beta\left(\frac{u^\eps - 1}{\eps}\right) \le (u^\eps)^{p-1},
\]
there exists a constant $C > 0$ such that
\[
\norm[\infty]{u^\eps} \le C \norm[q]{u^\eps}^{2q/(2q - N(p-2))}
\]
by Bonforte et al. \cite[Theorem 3.1]{BonforteGrilloVazquez}. Since $u^\eps$ is bounded in $L^q(\Omega)$ by the Sobolev imbedding theorem and Lemma \ref{Lemma 3}, the conclusion follows.
\end{proof}

By Lemma \ref{Lemma 2}, $(u^\eps - 1)_+^{p-1} \le A_0$ for some constant $A_0 > 0$. Let $\varphi_0 > 0$ be the solution of
\[
\left\{\begin{aligned}
- \Delta \varphi_0 & = A_0 && \text{in } \Omega\\[10pt]
\varphi_0 & = 0 && \text{on } \bdry{\Omega}.
\end{aligned}\right.
\]

\begin{lemma} \label{Lemma 4}
For $0 < \eps \le 1$,
\[
u^\eps(x) \le \varphi_0(x) \quad \forall x \in \Omega,
\]
in particular, $\set{u^\eps \ge 1} \subset \set{\varphi_0 \ge 1} \strictsubset \Omega$.
\end{lemma}

\begin{proof}
Since $\beta(t) \ge 0$ for all $t$,
\[
- \Delta u^\eps \le (u^\eps - 1)_+^{p-1} \le A_0 = - \Delta \varphi_0,
\]
so $u^\eps \le \varphi_0$ by the maximum principle.
\end{proof}

\begin{lemma} \label{Lemma 5}  There is a constant $C$ such
that for $r>0$ and $0 < \eps \le 1$, if $B_r(x_0) \subset\Omega$, then
\[
\max_{x \in B_{r/2}(x_0)}\, |\nabla u^\eps(x)| \le C/r
\]
\end{lemma}

\begin{proof}
Since $\beta(t) \le 2$ for all $t$,
\[
\Delta u^\eps \le \frac{1}{\eps}\, \beta\left(\frac{u^\eps - 1}{\eps}\right) \le \frac{2}{\eps}\, \goodchi_{\set{|u^\eps - 1| < \eps}}(x),
\]
and since $\beta(t) \ge 0$ for all $t$,
\[
- \Delta u^\eps \le (u^\eps - 1)_+^{p-1} \le A_0,
\]
so
\[
\pm \Delta u^\eps \le \max \set{2,A_0} \left(\frac{1}{\eps}\, \goodchi_{\set{|u^\eps - 1| < \eps}}(x) + 1\right).
\]
Since $u^\eps$ is also uniformly bounded in $L^2(\Omega)$ by 
Lemma \ref{Lemma 2}, the conclusion follows from the following
result of Caffarelli, Jerison and Kenig. 
\begin{proposition}[{\cite[Theorem 5.1]{MR1906591}}] \label{Proposition 3}
Suppose that $u$ is a Lipschitz continuous function on $B_1(0) \subset \R^N$ satisfying the distributional inequalities
\[
\pm \Delta u \le A \left(\frac{1}{\eps}\, \goodchi_{\set{|u - 1| < \eps}}(x) + 1\right)
\]
for some constants $A > 0$ and $0 < \eps \le 1$. Then there exists a constant $C > 0$, depending on $N$, $A$ and $\int_{B_1(0)} u^2\, dx$, but not on $\eps$, such that
\[
\max_{x \in B_{1/2}(0)}\, |\nabla u(x)| \le C.
\]
\end{proposition}
\end{proof}

\section{The limit is a generalized solution}\label{sec: generalized}


\begin{definition} \label{def: generalized}  We say a locally Lipschitz function $u$ defined
on $\Omega$ is a generalized solution to \eqref{1} if 
$u\in C^0(\bar{\Omega}) \cap 
 H^1_0(\Omega) \cap C^2(\Omega \setminus \bdry{\set{u > 1}})$ and satisfies 
\[
- \Delta u = (u - 1)_+^{p-1} \quad \mbox{on} \quad 
\Omega \setminus \bdry{\set{u > 1}}
\]
in the classical sense, and $u=0$ on $\partial \Omega$. Moreover, 
for all $\varphi \in C^1_0(\Omega,\R^N)$ such that $u \ne 1$
a.e.\! on the support of $\varphi$,
\[
\lim_{\delta^+ \searrow 0}\, \int_{\set{u = 1 + \delta^+}} \left(|\nabla u|^2 - 2\right) \varphi \cdot n^+\, d\sigma - \lim_{\delta^- \searrow 0}\, \int_{\set{u = 1 - \delta^-}} |\nabla u|^2\, \varphi \cdot n^-\, d\sigma = 0,
\]
where $n^\pm$ are the outward unit normals to $\bdry{\set{u > 1 \pm
    \delta^\pm}}$. (The sets $\set{u = 1 \pm \delta^\pm}$ are smooth
hypersurfaces for a.a.\! $\delta^\pm > 0$ by Sard's theorem and the
above limits are taken through such $\delta^\pm$.)
\end{definition}
In particular, a generalized solution $u$ to \eqref{1}
satisfies the free boundary condition in the classical sense on any
smooth portion of the free boundary $\bdry{\set{u > 1}}$.

Let $\eps_j \searrow 0$, let $u_j$ be the critical point of $J_{\eps_j}$ 
obtained in Lemma \ref{Lemma 7}, and set $c_j = c_{\eps_j} = J_{\eps_j}(u_j)$. 
\begin{lemma} \label{Lemma 6}
There exists a locally Lipschitz continuous function $u \in
C^0(\closure{\Omega}) \cap H^1_0(\Omega)$ 
such that, for a suitable sequence $\eps_j$, 
\begin{enumroman}
\item \label{6.i} $u_j \to u$ uniformly on $\closure{\Omega}$,
\item \label{6.ii} $u_j \to u$ strongly in $H^1_0(\Omega)$,
\item \label{6.iv} $J(u) \le \varliminf c_j \le \varlimsup c_j \le J(u) + \vol{\set{u = 1}}$.
\end{enumroman}
Moreover, $u$ is a nontrivial generalized solution of problem \eqref{1}.
\end{lemma}

\begin{proof}
First we prove \ref{6.i}. Since $\seq{u_j}$ is bounded in
$H^1_0(\Omega)$ by Lemma \ref{Lemma 3}, we may choose $\eps_j$ 
so that $u_j$ converges weakly in $H^1_0(\Omega)$ to some $u$.
Because, by Lemmas \ref{Lemma 2} and \ref{Lemma 5},  
$\seq{u_j}$ is uniformly bounded and uniformly locally Lipschitz,
we may also choose the sequence so that $u_j \to u$ uniformly on 
compact subsets of $\Omega$, and $u$ is locally Lipschitz continuous. 
Since $0 < u_j \le \varphi_0$ by Lemma \ref{Lemma 4}, we have
$0 \le u \le \varphi_0$, 
and hence $|u_j - u| \le \varphi_0$. Thus $u$ extends continuously to
$\closure{\Omega}$ with zero boundary values and $u_j \to u$ uniformly
on $\closure{\Omega}$.

Next we show that $u$ satisfies the equation $- \Delta u = (u -
1)_+^{p-1}$ in $\set{u \ne 1}$. Let $\varphi \in C^\infty_0(\set{u >
  1})$. Then $u \ge 1 + 2\, \eps$ on the support of $\varphi$ for some
$\eps > 0$. For all sufficiently large $j$, $\eps_j < \eps$ and $|u_j
- u| < \eps$ by \ref{6.i}. Then $u_j \ge 1 + \eps_j$ on the support of
$\varphi$, so testing
\begin{equation} \label{7}
- \Delta u_j = (u_j - 1)_+^{p-1} - \frac{1}{\eps_j}\, \beta\left(\frac{u_j - 1}{\eps_j}\right)
\end{equation}
with $\varphi$ gives
\[
\int_\Omega \nabla u_j \cdot \nabla \varphi\, dx = \int_\Omega (u_j - 1)^{p-1}\, \varphi\, dx.
\]
Passing to the limit gives
\begin{equation} \label{8}
\int_\Omega \nabla u \cdot \nabla \varphi\, dx = \int_\Omega (u - 1)^{p-1}\, \varphi\, dx
\end{equation}
since $u_j$ converges to $u$ weakly in $H^1_0(\Omega)$ and uniformly
in $\Omega$. This then holds for all $\varphi \in H^1_0(\set{u > 1})$
by density, and hence $u$ is a classical solution of $- \Delta u = (u
- 1)^{p-1}$ in $\set{u > 1}$. A similar argument shows that $u$ is
harmonic in $\set{u < 1}$.

Now we show that $u$ is harmonic in $\interior{\set{u \le 1}}$. Testing \eqref{7} with any nonnegative $\varphi \in C^\infty_0(\Omega)$ gives
\[
\int_\Omega \nabla u_j \cdot \nabla \varphi\, dx \le A_0 \int_\Omega \varphi\, dx
\]
since $\beta(t) \ge 0$ for all $t$ and $(u_j - 1)_+^{p-1} \le A_0$, and passing to the limit gives
\[
\int_\Omega \nabla u \cdot \nabla \varphi\, dx \le A_0 \int_\Omega \varphi\, dx.
\]
So
\begin{equation} \label{19}
- \Delta u \le A_0 \quad \text{in } \Omega
\end{equation}
in the weak sense. On the other hand, since $u$ is harmonic in $\set{u < 1}$, $\mu := \Delta (u - 1)_-$ is a nonnegative Radon measure supported on $\Omega \cap \bdry{\set{u < 1}}$ by Alt and Caffarelli \cite[Remark 4.2]{MR618549}, so
\begin{equation} \label{20}
- \Delta u = \mu \ge 0 \quad \text{in } \set{u \le 1}.
\end{equation}
It follows from \eqref{19} and \eqref{20} that $u \in W^{2,\, q}_\loc(\interior{\set{u \le 1}}),\, 1 < q < \infty$ and hence $\mu$ is actually supported on $\Omega \cap \bdry{\set{u < 1}} \cap \bdry{\set{u > 1}}$, so $u$ is harmonic in $\interior{\set{u \le 1}}$.

Since $u_j$ 
tends weakly to $u$ in $H^1_0(\Omega)$, $\norm{u} \le
\liminf \norm{u_j}$.  Thus to prove \ref{6.ii}, it suffices to show 
that $\limsup \norm{u_j} \le \norm{u}$.  The majorant $\varphi_0$ shows
that $\set{u_j\ge 1}$ is a fixed distance from $\partial \Omega$,
uniformly in $j$.  It follows from standard regularity
arguments that $u_j$ is uniformly in $C^2$ in a sufficiently
small, fixed neighborhood of $\partial \Omega$. 
Therefore, after replacing $u_j$ with a subsequence, we
may assume that $\partial u_j/\partial n \to \partial
u/\partial n$ uniformly on $\bdry{\Omega}$, where $n$ is the outward unit
normal. Multiplying \eqref{7} by $u_j - 1$, integrating by parts, and
noting that $\beta((t - 1)/\eps_j)\, (t - 1) \ge 0$ for all $t$ gives
\begin{equation} \label{9}
\int_\Omega |\nabla u_j|^2\, dx \le \int_\Omega (u_j - 1)_+^p\, dx - \int_{\bdry{\Omega}} \frac{\partial u_j}{\partial n}\, d\sigma \to \int_\Omega (u - 1)_+^p\, dx - \int_{\bdry{\Omega}} \frac{\partial u}{\partial n}\, d\sigma.
\end{equation}
Fix $0 < \eps < 1$. Taking $\varphi = (u - 1 - \eps)_+$ in \eqref{8} gives
\begin{equation} \label{10}
\int_{\set{u > 1 + \eps}} |\nabla u|^2\, dx = \int_\Omega (u - 1)_+^{p-1}\, (u - 1 - \eps)_+\, dx,
\end{equation}
and integrating $(u - 1 + \eps)_-\, \Delta u = 0$ over $\Omega$ gives
\begin{equation} \label{11}
\int_{\set{u < 1 - \eps}} |\nabla u|^2\, dx = - (1 - \eps) \int_{\bdry{\Omega}} \frac{\partial u}{\partial n}\, d\sigma.
\end{equation}
Adding \eqref{10} and \eqref{11}, and letting $\eps \searrow 0$ gives
\[
\int_\Omega |\nabla u|^2\, dx = \int_\Omega (u - 1)_+^p\, dx - \int_{\bdry{\Omega}} \frac{\partial u}{\partial n}\, d\sigma.
\]
This together with \eqref{9} gives
\[
\limsup \int_\Omega |\nabla u_j|^2\, dx \le \int_\Omega |\nabla u|^2\, dx.
\]

To prove \ref{6.iv}, write
\begin{multline*}
J_{\eps_j}(u_j) = \int_\Omega \left[\frac{1}{2}\, |\nabla u_j|^2 + \B\left(\frac{u_j - 1}{\eps_j}\right) \goodchi_{\set{u \ne 1}}(x) - \frac{1}{p}\, (u_j - 1)_+^p\right] dx\\[10pt]
+ \int_{\set{u = 1}} \B\left(\frac{u_j - 1}{\eps_j}\right) dx.
\end{multline*}
Since $\B((u_j - 1)/\eps_j)\, \goodchi_{\set{u \ne 1}}$ converges pointwise to $\goodchi_{\set{u > 1}}$ and is bounded by $1$, the first integral converges to $J(u)$ by \ref{6.i} and \ref{6.ii}. Since $0 \le \B(t) \le 1$ for all $t$,
\[
0 \le \int_{\set{u = 1}} \B\left(\frac{u_j - 1}{\eps_j}\right) dx \le \vol{\set{u = 1}}.
\]
\ref{6.iv} follows.

By \ref{6.iv} and \eqref{22},
\[
J(u) + \vol{\set{u = 1}} \ge \frac{\rho^2}{3} > 0
\]
and hence $u$ is nontrivial.

Finally we show that $u$ satisfies the generalized free boundary
condition, i.e., for all $\varphi \in C^1_0(\Omega,\R^N)$ such that $u
\ne 1$ a.e.\! on the support of $\varphi$,
\begin{equation} \label{23}
\lim_{\delta^+ \searrow 0}\, \int_{\set{u = 1 + \delta^+}} \left(2 - |\nabla u|^2\right) \varphi \cdot n\, d\sigma - \lim_{\delta^- \searrow 0}\, \int_{\set{u = 1 - \delta^-}} |\nabla u|^2\, \varphi \cdot n\, d\sigma = 0,
\end{equation}
where $n$ is the outward unit normal to $\set{1 - \delta^- < u < 1 + \delta^+}$. Multiplying \eqref{7} by $\nabla u_j \cdot \varphi$ and integrating over $\set{1 - \delta^- < u < 1 + \delta^+}$ gives
\begin{eqnarray*}
0 & = & \int_{\set{1 - \delta^- < u < 1 + \delta^+}} \left[- \Delta u_j + \frac{1}{\eps_j}\, \beta\left(\frac{u_j - 1}{\eps_j}\right) - (u_j - 1)_+^{p-1}\right] \nabla u_j \cdot \varphi\, dx\\[10pt]
& = & \int_{\set{1 - \delta^- < u < 1 + \delta^+}} \bigg[\divg \left(\frac{1}{2}\, |\nabla u_j|^2\, \varphi - (\nabla u_j \cdot \varphi)\, \nabla u_j\right) + \nabla u_j\, D\varphi \cdot \nabla u_j\\[10pt]
&& - \frac{1}{2}\, |\nabla u_j|^2\, \divg \varphi + \nabla \B\left(\frac{u_j - 1}{\eps_j}\right) \cdot \varphi - \frac{1}{p}\, \nabla (u_j - 1)_+^p \cdot \varphi\bigg]\, dx\\[10pt]
& = & \frac{1}{2} \int_{\set{u = 1 + \delta^+} \bigcup \set{u = 1 - \delta^-}} \left[|\nabla u_j|^2\, \varphi - 2\, (\nabla u_j \cdot \varphi)\, \nabla u_j + 2\, \B\left(\frac{u_j - 1}{\eps_j}\right) \varphi\right] \cdot n\, d\sigma\\[10pt]
&& - \frac{1}{p} \int_{\set{u = 1 + \delta^+} \bigcup \set{u = 1 - \delta^-}} (u_j - 1)_+^p\, \varphi \cdot n\, d\sigma\\[10pt]
&& - \int_{\set{1 - \delta^- < u < 1 + \delta^+}} \left[\B\left(\frac{u_j - 1}{\eps_j}\right) - \frac{1}{p}\, (u_j - 1)_+^p\right] \divg \varphi\, dx\\[10pt]
&& + \int_{\set{1 - \delta^- < u < 1 + \delta^+}} \left(\nabla u_j\, D\varphi \cdot \nabla u_j - \frac{1}{2}\, |\nabla u_j|^2\, \divg \varphi\right) dx\\[10pt]
& =: & \frac{I_1}{2} - \frac{I_2}{p} - I_3 + I_4.
\end{eqnarray*}
Since $u_j \to u$ uniformly on $\closure{\Omega}$, strongly in $H^1_0(\Omega)$, and locally in $C^1(\set{u \ne 1})$,
\begin{eqnarray*}
I_1 & \to & \int_{\set{u = 1 + \delta^+} \bigcup \set{u = 1 - \delta^-}} \left(|\nabla u|^2\, \varphi - 2\, (\nabla u \cdot \varphi)\, \nabla u\right) \cdot n\, d\sigma + \int_{\set{u = 1 + \delta^+}} 2\, \varphi \cdot n\, d\sigma\\[10pt]
& = & \int_{\set{u = 1 + \delta^+}} \left(2 - |\nabla u|^2\right) \varphi \cdot n\, d\sigma - \int_{\set{u = 1 - \delta^-}} |\nabla u|^2\, \varphi \cdot n\, d\sigma
\end{eqnarray*}
since $n = \pm \nabla u/|\nabla u|$ on $\set{u = 1 \pm \delta^\pm}$, and
\[
I_2 \to \int_{\set{u = 1 + \delta^+}} (u - 1)_+^p\, \varphi \cdot n\, d\sigma = (\delta^+)^p \int_{\set{u = 1 + \delta^+}} \varphi \cdot n\, d\sigma = (\delta^+)^p \int_{\set{u < 1 + \delta^+}} \divg \varphi\, dx,
\]
which goes to zero as $\delta^+ \searrow 0$. Since $|\B((u_j - 1)/\eps_j)| \le 1$,
\begin{eqnarray*}
|I_3| & \le & \int_{\set{1 - \delta^- < u < 1 + \delta^+}} \left[1 + \frac{1}{p}\, (u_j - 1)_+^p\right] |\divg \varphi|\, dx\\[10pt]
& \to & \int_{\set{1 - \delta^- < u < 1 + \delta^+}} \left[1 + \frac{1}{p}\, (u - 1)_+^p\right] |\divg \varphi|\, dx,
\end{eqnarray*}
and
\[
I_4 \to \int_{\set{1 - \delta^- < u < 1 + \delta^+}} \left(\nabla u\, D\varphi \cdot \nabla u - \frac{1}{2}\, |\nabla u|^2\, \divg \varphi\right) dx.
\]
The last two integrals go to zero as $\delta^\pm \searrow 0$ since $\vol{\set{u = 1} \cap \supp \varphi} = 0$, so first letting $j \to \infty$ and then letting $\delta^\pm \searrow 0$ gives \eqref{23}.
\end{proof}

\section{The Nehari manifold and non-degeneracy} \label{sec: Nehari nondegeneracy}

\begin{lemma}\label{lem: gen.Nehari}  Every nonzero
generalized solution $u$ 
to \eqref{1} belongs to the Nehari manifold $\M$ and satisfies
$J(u) >0$.  
\end{lemma}
\begin{proof}
If $u$ is a generalized solution of problem \eqref{1}, then by the
maximum principle, the set $\set{u < 1}$ is connected and either $u >
0$ everywhere or $u$ vanishes identically. If $u \le 1$ everywhere,
then $u$ is harmonic in $\Omega$ and hence vanishes identically
again. So if $u$ is nontrivial, then $u > 0$ in $\Omega$ and $u > 1$
in a nonempty open subset of $\Omega$, where it satisfies $- \Delta u
= (u - 1)^{p-1}$.  As in the proof of Lemma \ref{Lemma 6}  \ref{6.ii},
multiplying this equation by $u - 1$ and integrating
over the set $\set{u > 1}$ shows that $u$ lies on $\M$. Finally,
if $u\in \M$, then 
\[
J(u) = \frac{1}{2} \int_{\set{u < 1}} |\nabla u|^2\, dx + \left(\frac{1}{2} - \frac{1}{p}\right) \int_{\set{u > 1}} |\nabla u|^2\, dx + \vol{\set{u > 1}} > 0,
\]
where $\vol{\cdot}$ denotes the Lebesgue measure in $\R^N$. 
\end{proof}

For $u\in H_0^1(\Omega)$, set
\[
u^+ = (u - 1)_+, \quad u^- = 1 - (u - 1)_-; \qquad u = u^- + u^+.
\]
Let 
\[
W =\set{u \in H^1_0(\Omega) : u^\pm \neq 0}
\]
Then $\M \subset W$, and for $u\in W$, we define
the curve
\[
\zeta_u(s) = \begin{cases}
(1 + s)\, u^-, & s \in [-1,0]\\[5pt]
u^- + s\, u^+, & s \in (0,\infty),
\end{cases}
\]
which passes through $u$ at $s = 1$. For $s \in [-1,0]$,
\[
J(\zeta_u(s)) = \frac{(1 + s)^2}{2} \int_{\set{u < 1}} |\nabla u|^2\, dx
\]
is increasing in $s$.  There is a discontinuity in $J$ at $s=0$:
\[
\lim_{s \searrow 0}\, J(\zeta_u(s)) = J(\zeta_u(0)) + \vol{\set{u > 1}} > J(\zeta_u(0)).
\]
For $s \in (0,\infty)$, 
\begin{multline} \label{30}
J(\zeta_u(s)) = \frac{1}{2} \int_{\set{u < 1}} |\nabla u|^2\, dx + \frac{s^2}{2} \int_{\set{u > 1}} |\nabla u|^2\, dx - \frac{s^p}{p} \int_{\set{u > 1}} (u - 1)^p\, dx\\[7.5pt]
+ \vol{\set{u > 1}}
\end{multline}
and
\[
\frac{d}{ds}\, J(\zeta_u(s)) = s \left[\int_{\set{u > 1}} |\nabla u|^2\, dx - s^{p-2} \int_{\set{u > 1}} (u - 1)^p\, dx\right].
\]
Define 
\[
s_u = \left[\frac{\dint_{\set{u > 1}} |\nabla u|^2\, dx}{\dint_{\set{u > 1}} (u - 1)^p\, dx}\right]^{1/(p-2)}.
\]
Thus we see that $J(\zeta_u(s))$ increases for $s \in [-1,s_u)$, 
attains its maximum at $s = s_u$ and decreases for $s \in (s_u,\infty)$, and
\begin{equation} \label{26}
\lim_{s \to \infty}\, J(\zeta_u(s)) = - \infty.
\end{equation}

\begin{proposition} \label{Proposition 9}
We have
\begin{equation} \label{31}
c^* \le \inf_{u \in \M}\, J(u).
\end{equation}
If $u \in \M$ and $J(u) = c^*$, then $u$ is a mountain pass point of $J$.
\end{proposition}

\begin{proof}
For each $u \in \M$, \eqref{26} implies that 
we may choose $\bar{s} > 1$ sufficiently large
that that $J(\zeta_u(\bar{s})) < 0$.  Note
that $s_u=1$.  Therefore, 
\[
\gamma_u(t) = \zeta_u((\bar{s} + 1)\, t - 1), \quad t \in [0,1]
\]
defines a path $\gamma_u \in \Gamma$ such that
\[
\max_{v \in \gamma_u([0,1])}\, J(v) = J(\zeta_u(s_u)) = J(u),
\]
so $c^* \le J(u)$. \eqref{31} follows.

Now suppose $J(u) = c^*$ and let $U$ be a neighborhood of $u$. The
path $\gamma_u$ passes through $u$ at $t = 2/(\bar{s} + 1) =: \bar{t}$ and
$J(\gamma_u(t)) < c$ for $t \ne \bar{t}$. By the continuity of $\gamma_u$,
there exist $0 < t^- < \bar{t} < t^+ < 1$ such that $\gamma_u(t^\pm) \in
U$, in particular, the set $\set{v \in U : J(v) < c}$ is nonempty. If
it is path connected, then this set contains a path $\eta$ joining
$\gamma_u(t^\pm)$, and reparametrizing $\restr{\gamma_u}{[0,t^-]} \cup
\eta \cup \restr{\gamma_u}{[t^+,1]}$ gives a path in $\Gamma$ on which
$J < c^*$, contradicting the definition of $c^*$. So the set is not path
connected, and $u$ is a mountain pass point of $J$.
\end{proof}

For $u \in W$, $\zeta_u$ intersects $\M$ exactly at one point, namely, 
where $s = s_u$, and $s_u = 1$ if $u \in \M$. So we can define a 
continuous projection $\pi : W \to \M$ by
\[
\pi(u) = \zeta_u(s_u) = u^- + s_u\, u^+.
\]

\begin{lemma} \label{Lemma 8}
For $u \in W$,
\[
J(\pi(u)) = \frac{1}{2} \int_{\set{u < 1}} |\nabla u|^2\, dx + \left(\frac{1}{2} - \frac{1}{p}\right) s_u^2 \int_{\set{u > 1}} |\nabla u|^2\, dx + \vol{\set{u > 1}}.
\]
In particular, for $u \in \M$, since $\pi(u) = u$, 
\[
J(u) = \frac{1}{2} \int_{\set{u < 1}} |\nabla u|^2\, dx + \left(\frac{1}{2} - \frac{1}{p}\right) \int_{\set{u > 1}} |\nabla u|^2\, dx + \vol{\set{u > 1}}.
\]
\end{lemma}

\begin{proof}
For $u \in W$, $J(\pi(u))$ is given by \eqref{30} with $s = s_u$, and
\[
s_u^2 \int_{\set{u > 1}} |\nabla u|^2\, dx = s_u^p \int_{\set{u > 1}} (u - 1)^p\, dx. \QED
\]
\end{proof}

\begin{proposition} \label{Proposition 2}
If $u$ is a locally Lipschitz continuous minimizer of $\restr{J}{\M}$, then $u$ is nondegenerate (Definition \ref{def: degenerate}).
\end{proposition}

\begin{proof} Suppose that $B_r(x_0) \subset \set{x\in \Omega: u(x)>1}$ and
there is $x_1 \in \partial B_r(x_0)$ such that $u(x_1) = 1$. 
Define 
\[
v(y) = \frac1r(u(x_0+ ry) - 1).
\]
Our goal is to show that 
\[
\alpha : = v(0)\ge c >0
\]
We begin by oberving that 
\begin{equation} \label{35}
0 < v(y) = \frac1r (u(x_0+ry) - u(x_1)) \le \frac{L}{r}|x_0 - x_1 +ry| 
\le 2L \quad \forall y \in B_1(0),
\end{equation}
where $L$ is the Lipschitz constant of $u$ in $\set{u \ge 1}$.  Therefore,
\[
- \Delta v = r^p\, v^{p-1} \quad \text{in } B_1(0),
\]
Define $h$ by 
\[
-\Delta h = r^p\, v^{p-1} \ \mbox{in} \ B_1(0), \quad
h=0 \ \mbox{on} \ \partial B_1(0).
\]
Then $|h| \le C L^{p-1} r^p$ and applying the Harnack inequality
to $v-h + \max h$, there is a constant $C$ depending on $L$ and
dimension such that 
\[
v(y) \le C\, (\alpha + r^p) \quad \forall y \in B_{2/3}(0),
\]

Take a smooth cutoff function $\psi : B_1(0) \to [0,1]$ such that $\psi = 0$ in $\closure{B_{1/3}(0)}$, $0 < \psi < 1$ in $B_{2/3}(0) \setminus \closure{B_{1/3}(0)}$ and $\psi = 1$ in $B_1(0) \setminus B_{2/3}(0)$, let
\[
w(y) = \begin{cases}
\min \set{v(y),C\, (\alpha + r^p)\, \psi(y)}, & y \in B_{2/3}(0)\\[7.5pt]
v(y), & \mbox{otherwise}, 
\end{cases}
\]
and set $z(x) = 1 + r w((x-x_0)/r)$. Since $u$ is a minimizer of $\restr{J}{\M}$,
\[
J(u) \le J(\pi(z)).
\]
Since $z^- = u^-$, $z = 1$ in $\closure{B_{r/3}(x_0)}$, and $\set{z > 1} = \set{u > 1} \setminus \closure{B_{r/3}(x_0)}$, Lemma \ref{Lemma 8} implies
this inequality can be rewritten as 
\[
\left(\frac{1}{2} - \frac{1}{p}\right) \int_{\set{u > 1}} |\nabla u|^2\, dx + \vol{B_{r/3}(x_0)} \le \left(\frac{1}{2} - \frac{1}{p}\right) s_z^2 \int_{\set{u > 1}} |\nabla z|^2\, dx \, .
\]
Let $y = (x-x_0)/r$ and define
\[
\D := \set{x \in B_{2r/3}(x_0) : v(y) > C\, (\alpha + r^p)\, \psi(y)}
\]
Because $z = u$ outside $\D$, this last inequality implies
\begin{equation} \label{36}
s_z^2 \int_\D |\nabla z|^2\, dx + \left(s_z^2 - 1\right) \int_{\set{u > 1} \setminus \D} |\nabla u|^2\, dx \ge \frac{2p}{p - 2} \vol{B_{1/3}(0)} r^N.
\end{equation}

Since $\set{z > 1} = \set{u > 1} \setminus \closure{B_{r/3}(x_0)}$ and
$z = 1$ in $\closure{B_{r/3}(x_0)}$,
\[
s_z^{p-2} = \frac{\dint_{\set{z > 1}} |\nabla z|^2\, dx}{\dint_{\set{z > 1}} (z - 1)^p\, dx} = \frac{\dint_{\set{u > 1}} |\nabla z|^2\, dx}{\dint_{\set{u > 1}} (z - 1)^p\, dx}.
\]
Since $z = u$ in $\set{u > 1} \setminus \D$, we have 
\[
s_z^{p-2} \le
\frac{\dint_{\set{u > 1}} |\nabla u|^2\, dx + \dint_\D |\nabla z|^2\, dx}{\dint_{\set{u > 1}} (u - 1)^p\, dx - \dint_\D (u - 1)^p\, dx}
= 
\frac{A_1 + \int_\D |\nabla z|^2 \, dx}
{A_1 - \int_\D (u-1)^p \, dx},
\]
where, since $u\in \M$,
\[
A_1  =  \int_{\set{u>1}} |\nabla u|^2 \, dx = \int_{\set{u>1}} (u-1)^p \, dx
\]


It follows as in \eqref{35}, 
$0 < u - 1 < 2L\, r$ in $\D$, and $\vol{\D} = \O(r^N)$ as $r \to 0$.
Thus
\[
\int_\D (u - 1)^p\, dx = \O(r^{p+N}).
\]
It follows that
\begin{equation} \label{37}
s_z^{p-2} \le 1 + \frac1{A_1} 
\dint_\D |\nabla z|^2\, dx + \O(r^{p+N}).
\end{equation}
We have
\begin{equation} \label{28}
\int_\D |\nabla z|^2\, dx = C^2\, (\alpha + r^p)^2\, r^N \int_{\set{y : x \in \D}} |\nabla \psi|^2\, dy.
\end{equation}
The right-hand side is $\O(r^N)$ since $0 < \alpha < 2L$ by \eqref{35}.
Consequently, so \eqref{37} gives
\[
s_z^2 \le 1 + \frac{2}{(p - 2)A_1}\dint_\D |\nabla z|^2\, dx  + \O(r^{q+N}),
\]
where $q = \min \set{p,N} \ge 2$. Using this estimate in \eqref{36} now gives
\[
\frac{1}{r^N} \int_\D |\nabla z|^2\, dx + \O(r^q) \ge 2 \vol{B_{1/3}(0)}.
\]
In view of  \eqref{28}, we find that there are 
$r_0,\, c > 0$ such that $r \le r_0$  
implies $\alpha \ge c$.  This concludes the proof of nondegeneracy.
\end{proof}

\section{Proof of the main theorem and further boundary regularity} 
\label{sec: final}

\begin{proof}[Proof of Theorem \ref{Theorem 1}]
We can now conclude the proof of our main theorem. 
Let $u$ be the nontrivial generalized solution of problem \eqref{1} obtained 
in Lemma \ref{Lemma 6}. Since $u \in \M$,
\[
\inf_\M\, J \le J(u).
\]
By Lemma \ref{Lemma 6} \ref{6.iv}, Lemma \ref{Lemma 1}, and 
Proposition \ref{Proposition 9}, we also have
\[
J(u) \le \varliminf c_j \le \varlimsup c_j \le c^* \le \inf_\M\, J.
\]
So
\[
J(u) = c^* = \inf_\M\, J
\]
and $c_j \to c$. Then $u$ is a mountain pass point of $J$ by Proposition \ref{Proposition 9}, minimizes $\restr{J}{\M}$, and is therefore nondegenerate by Proposition \ref{Proposition 2}.
\end{proof}

\begin{definition}
We say that $u \in C(\Omega)$ satisfies the free boundary condition
$|\nabla u^+|^2 - |\nabla u^-|^2 = 2$ in the viscosity sense if
whenever there exist a point $x_0 \in \bdry{\set{u > 1}}$, a ball $B
\subset \set{u > 1}$ (resp. $\interior{\set{u \le 1}}$) with $x_0 \in
\bdry{B}$, and $\alpha$ (resp. $\gamma$) $\ge 0$ such that
\[
u(x) \ge \alpha \ip{x - x_0}{\nu}_+ + \o(|x - x_0|) \hquad (\text{resp. } u(x) \le - \gamma \ip{x - x_0}{\nu}_- + \o(|x - x_0|))
\]
in $B$, where $\nu$ is the interior (resp. exterior) unit normal to $\bdry{B}$ at $x_0$, we have
\[
u(x) < - \gamma \ip{x - x_0}{\nu}_- + \o(|x - x_0|) \hquad (\text{resp. } u(x) > \alpha \ip{x - x_0}{\nu}_+ + \o(|x - x_0|))
\]
in $B^c$ for any $\gamma$ (resp. $\alpha$) $\ge 0$ such that $\alpha^2 - \gamma^2 >$ (resp. $<$) $2$.
\end{definition}

\begin{definition}
We say that the point $x_0 \in \bdry{\set{u > 1}}$ is regular if there
exists a unit vector $\nu \in \R^N$, called the interior unit normal
to the free boundary $\bdry{\set{u > 1}}$ at $x_0$ in the measure
theoretic sense, such that
\[
\lim_{r \to 0}\, \frac{1}{r^N} \int_{B_r(x_0)} \abs{\goodchi_{\set{u > 1}}(x) - \goodchi_{\set{\ip{x - x_0}{\nu} > 0}}(x)} dx = 0.
\]
\end{definition}

\begin{proof}[Proof of Corollary \ref{Corollary 1}]
Let $\eps_j \searrow 0$ be a suitable sequence and let $u_j$ be the solution 
of \eqref{2} obtained in Lemma \ref{Lemma 7}. Then $u_j$ is uniformly bounded 
on $\Omega$ by Lemma \ref{Lemma 2} and $u_j$ converges uniformly on 
$\Omega$ to $u$ by Lemma \ref{Lemma 6} \ref{6.i}. Set
\[
f_j(x) = - (u_j(x) - 1)_+^{p-1}, \quad f(x) = - (u(x) - 1)_+^{p-1}, \quad x \in \Omega.
\]
Since $p > 2$, $f_j \to f$ uniformly on $\Omega$. Since, by Theorem \ref{Theorem 1},
$u$ is nondegenerate, the corollary follows from now follow from 
Corollaries 7.1, 7.2 and Theorem 9.2, respectively, of Lederman and Wolanski \cite{MR2281453}.
\end{proof}

We expect that, at least in dimension 2, the 
free boundary has a measure-theoretic
normal at all points and hence is smooth.  But this is an open
question, and it would even be nice to show that the
measure-theoretic normal exists at ``most'' free boundary points.

\def\cdprime{$''$}


\begin{thebibliography}{10}

\bibitem{MR618549}
H.~W. Alt and L.~A. Caffarelli.
\newblock Existence and regularity for a minimum problem with free boundary.
\newblock {\em J. Reine Angew. Math.}, 325:105--144, 1981.

\bibitem{MR732100}
Hans~Wilhelm Alt, Luis~A. Caffarelli, and Avner Friedman.
\newblock Variational problems with two phases and their free boundaries.
\newblock {\em Trans. Amer. Math. Soc.}, 282(2):431--461, 1984.

\bibitem{MR1044809}
H.~Berestycki, L.~A. Caffarelli, and L.~Nirenberg.
\newblock Uniform estimates for regularization of free boundary problems.
\newblock In {\em Analysis and partial differential equations}, volume 122 of
  {\em Lecture Notes in Pure and Appl. Math.}, pages 567--619. Dekker, New
  York, 1990.

\bibitem{BonforteGrilloVazquez}
Matteo Bonforte, Gabriele Grillo, and Juan~Luis Vazquez.
\newblock Quantitative bounds for subcritical semilinear elliptic equations.
\newblock preprint.

\bibitem{MR2145284}
Luis Caffarelli and Sandro Salsa.
\newblock {\em A geometric approach to free boundary problems}, volume~68 of
  {\em Graduate Studies in Mathematics}.
\newblock American Mathematical Society, Providence, RI, 2005.

\bibitem{MR861482}
Luis~A. Caffarelli.
\newblock A {H}arnack inequality approach to the regularity of free boundaries.
\newblock {\em Comm. Pure Appl. Math.}, 39(S, suppl.):S41--S45, 1986.
\newblock Frontiers of the mathematical sciences: 1985 (New York, 1985).

\bibitem{MR990856}
Luis~A. Caffarelli.
\newblock A {H}arnack inequality approach to the regularity of free boundaries.
  {I}. {L}ipschitz free boundaries are {$C^{1,\alpha}$}.
\newblock {\em Rev. Mat. Iberoamericana}, 3(2):139--162, 1987.

\bibitem{MR1029856}
Luis~A. Caffarelli.
\newblock A {H}arnack inequality approach to the regularity of free boundaries.
  {III}.\ {E}xistence theory, compactness, and dependence on {$X$}.
\newblock {\em Ann. Scuola Norm. Sup. Pisa Cl. Sci. (4)}, 15(4):583--602
  (1989), 1988.

\bibitem{MR973745}
Luis~A. Caffarelli.
\newblock A {H}arnack inequality approach to the regularity of free boundaries.
  {II}. {F}lat free boundaries are {L}ipschitz.
\newblock {\em Comm. Pure Appl. Math.}, 42(1):55--78, 1989.

\bibitem{MR587175}
Luis~A. Caffarelli and Avner Friedman.
\newblock Asymptotic estimates for the plasma problem.
\newblock {\em Duke Math. J.}, 47(3):705--742, 1980.

\bibitem{MR1906591}
Luis~A. Caffarelli, David Jerison, and Carlos~E. Kenig.
\newblock Some new monotonicity theorems with applications to free boundary
  problems.
\newblock {\em Ann. of Math. (2)}, 155(2):369--404, 2002.

\bibitem{MR2082392}
Luis~A. Caffarelli, David Jerison, and Carlos~E. Kenig.
\newblock Global energy minimizers for free boundary problems and full
  regularity in three dimensions.
\newblock In {\em Noncompact problems at the intersection of geometry,
  analysis, and topology}, volume 350 of {\em Contemp. Math.}, pages 83--97.
  Amer. Math. Soc., Providence, RI, 2004.

\bibitem{MR2572253}
Daniela De~Silva and David Jerison.
\newblock A singular energy minimizing free boundary.
\newblock {\em J. Reine Angew. Math.}, 635:1--21, 2009.

\bibitem{MR1644436}
M.~Flucher and J.~Wei.
\newblock Asymptotic shape and location of small cores in elliptic
  free-boundary problems.
\newblock {\em Math. Z.}, 228(4):683--703, 1998.

\bibitem{MR1009785}
Avner Friedman.
\newblock {\em Variational principles and free-boundary problems}.
\newblock Robert E. Krieger Publishing Co. Inc., Malabar, FL, second edition,
  1988.

\bibitem{MR1360544}
Avner Friedman and Yong Liu.
\newblock A free boundary problem arising in magnetohydrodynamic system.
\newblock {\em Ann. Scuola Norm. Sup. Pisa Cl. Sci. (4)}, 22(3):375--448, 1995.

\bibitem{MR812787}
Helmut Hofer.
\newblock A geometric description of the neighbourhood of a critical point
  given by the mountain-pass theorem.
\newblock {\em J. London Math. Soc. (2)}, 31(3):566--570, 1985.

\bibitem{MR0440187}
D.~Kinderlehrer and L.~Nirenberg.
\newblock Regularity in free boundary problems.
\newblock {\em Ann. Scuola Norm. Sup. Pisa Cl. Sci. (4)}, 4(2):373--391, 1977.

\bibitem{MR1664689}
Claudia Lederman and Noemi Wolanski.
\newblock Viscosity solutions and regularity of the free boundary for the limit
  of an elliptic two phase singular perturbation problem.
\newblock {\em Ann. Scuola Norm. Sup. Pisa Cl. Sci. (4)}, 27(2):253--288
  (1999), 1998.

\bibitem{MR2281453}
Claudia Lederman and Noemi Wolanski.
\newblock A two phase elliptic singular perturbation problem with a forcing
  term.
\newblock {\em J. Math. Pures Appl. (9)}, 86(6):552--589, 2006.

\bibitem{MR3012848}
Kanishka Perera and Martin Schechter.
\newblock {\em Topics in critical point theory}, volume 198 of {\em Cambridge
  Tracts in Mathematics}.
\newblock Cambridge University Press, Cambridge, 2013.

\bibitem{MR1932180}
Masataka Shibata.
\newblock Asymptotic shape of a least energy solution to an elliptic
  free-boundary problem with non-autonomous nonlinearity.
\newblock {\em Asymptot. Anal.}, 31(1):1--42, 2002.

\bibitem{MR0412637}
R.~Temam.
\newblock A non-linear eigenvalue problem: the shape at equilibrium of a
  confined plasma.
\newblock {\em Arch. Rational Mech. Anal.}, 60(1):51--73, 1975/76.

\bibitem{MR0602544}
R.~Temam.
\newblock Remarks on a free boundary value problem arising in plasma physics.
\newblock {\em Comm. Partial Differential Equations}, 2(6):563--585, 1977.

\bibitem{MR1620644}
Georg~S. Weiss.
\newblock Partial regularity for weak solutions of an elliptic free boundary
  problem.
\newblock {\em Comm. Partial Differential Equations}, 23(3-4):439--455, 1998.

\bibitem{MR1759450}
Georg~Sebastian Weiss.
\newblock Partial regularity for a minimum problem with free boundary.
\newblock {\em J. Geom. Anal.}, 9(2):317--326, 1999.

\end{thebibliography}
\end{document}